\documentclass[12pt]{amsart}
\usepackage{amsmath,amsthm,amsfonts,amssymb,mathrsfs}
\date{\today}

\usepackage{hyperref}
\input{xy}
 \xyoption{all}
 \xyoption{arc}

\newtheorem{theorem}{Theorem}

\newtheorem{corollary}[theorem]{Corollary}

\theoremstyle{definition}
\newtheorem{definition}[theorem]{Definition}

\begin{document}

\title[Symmetric inverse topological
semigroups of finite rank $\leqslant n$] {Symmetric inverse
topological semigroups\\ of finite rank $\leqslant n$}
\author{Oleg~Gutik}
\address{Department of Mechanics and Mathematics, Ivan Franko Lviv
National University, Universytetska 1, Lviv, 79000, Ukraine}
\email{o\_\,gutik@franko.lviv.ua, ovgutik@yahoo.com}

\author{Andriy~Reiter}
\address{Department of Mechanics and Mathematics, Ivan Franko Lviv
National University, Universytetska 1, Lviv, 79000, Ukraine}
\email{reiter@i.ua, reiter\_\,andriy@yahoo.com}

\keywords{Topological semigroup, topological inverse semigroup,
symmetric inverse semigroup of finite transformations,
algebraically $h$-closed semigroup, absolutely $H$-closed
semigroup, $H$-closed semigroup, annihilating semigroup
homomorphism, Bohr compactification}

\subjclass[2000]{Primary 22A15, 20M20. Secondary 20M18, 54H15}

\begin{abstract}
We establish topological properties of the symmetric inverse
topological semigroup of finite transformations
$\mathscr{I}_\lambda^n$ of the rank $\leqslant n$. We show that
the topological inverse semigroup $\mathscr{I}_\lambda^n$ is
algebraically $h$-closed in the class of topological inverse
semigroups. Also we prove that a topological semigroup $S$ with
countably compact square $S\times S$ does not contain the
semigroup $\mathscr{I}_\lambda^n$ for infinite cardinal $\lambda$
and show that the Bohr compactification of an infinite topological
symmetric inverse semigroup of finite transformations
$\mathscr{I}_\lambda^n$ of the rank $\leqslant n$ is the trivial
semigroup.
\end{abstract}

\maketitle



In this paper all topological spaces will be assumed to be
Hausdorff. We shall follow the terminology of~\cite{CHK, CP,
Engelking1989, Petrich1984}. If $A$ is a subset of a topological
space $X$, then we denote the closure of the set $A$ in $X$ by
$\operatorname{cl}_X(A)$. By $\omega$ we denote the first infinite
cardinal.

A semigroup $S$ is called an \emph{inverse semigroup} if every $a$
in $S$ possesses an unique inverse, i.e. if there exists an unique
element $a^{-1}$ in $S$ such that
\begin{equation*}
    aa^{-1}a=a \qquad \mbox{and} \qquad a^{-1}aa^{-1}=a^{-1}.
\end{equation*}
A map which associates to any element of an inverse semigroup its
inverse is called the \emph{inversion}.

A {\it topological} ({\it inverse}) {\it semigroup} is a
topological space together with a continuous semigroup operation
(and an~inversion, respectively). Obviously, the inversion defined
on a topological inverse semigroup is a homeomorphism. If $S$ is
a~semigroup (an~inverse semigroup) and $\tau$ is a topology on $S$
such that $(S,\tau)$ is a topological (inverse) semigroup, then we
shall call $\tau$ a \emph{semigroup} (\emph{inverse})
\emph{topology} on $S$.

If $S$ is a~semigroup, then by $E(S)$ we denote the band (the
subset of all idempotents) of $S$. On  the set of idempotents
$E(S)$ there exists a natural partial order: $e\leqslant f$
\emph{if and only if} $ef=fe=e$.

Let $X$ be a set of cardinality $\lambda\geqslant 1$. Without loss
of generality we can identify the set $X$ with the cardinal
$\lambda$. A function $\alpha$ mapping a subset $Y$ of $X$ into
$X$ is called a \emph{partial transformation} of $X$. In this case
the set $Y$ is called the \emph{domain} of $\alpha$ and is denoted
by $\operatorname{dom}\alpha$. Also, the set $\{ x\in X\mid
y\alpha=x \mbox{ for some } y\in Y\}$ is called the \emph{range}
of $\alpha$ and is denoted by $\operatorname{ran}\alpha$. The
cardinality of $\operatorname{ran}\alpha$ is called the
\emph{rank} of $\alpha$ and denoted by
$\operatorname{rank}\alpha$. For convenience we denote by
$\varnothing$ the empty transformation, that is a partial mapping
with
$\operatorname{dom}\varnothing=\operatorname{ran}\varnothing=\varnothing$.

Let $\mathscr{I}(X)$ denote the set of all partial one-to-one
transformations of $X$ together with the following semigroup
operation:
\begin{equation*}
    x(\alpha\beta)=(x\alpha)\beta \quad \mbox{if} \quad
    x\in\operatorname{dom}(\alpha\beta)=\{
    y\in\operatorname{dom}\alpha\mid
    y\alpha\in\operatorname{dom}\beta\}, \qquad \mbox{for} \quad
    \alpha,\beta\in\mathscr{I}(X).
\end{equation*}
The semigroup $\mathscr{I}(X)$ is called the \emph{symmetric
inverse semigroup} over the set $X$~(see \cite{CP}). The symmetric
inverse semigroup was introduced by V.~V.~Wagner~\cite{Wagner1952}
and it plays a major role in the theory of semigroups.

Put
\begin{equation*}
\mathscr{I}_\lambda^\infty=\{ \alpha\in\mathscr{I}(X)\mid
\operatorname{rank}\alpha \;\mbox{ is finite}\}, \quad \mbox{and}
\quad \mathscr{I}_\lambda^n=\{ \alpha\in\mathscr{I}(X)\mid
\operatorname{rank}\alpha\leqslant n\},
\end{equation*}
for $n=1,2,3,\ldots$. Obviously, $\mathscr{I}_\lambda^\infty$ and
$\mathscr{I}_\lambda^n$ ($n=1,2,3,\ldots$) are inverse semigroups,
$\mathscr{I}_\lambda^\infty$ is an ideal of $\mathscr{I}(X)$, and
$\mathscr{I}_\lambda^n$ is an ideal of
$\mathscr{I}_\lambda^\infty$, for each $n=1,2,3,\ldots$. Further,
we shall call the semigroup $\mathscr{I}_\lambda^\infty$ the
\emph{symmetric inverse semigroup of finite transformations} and
$\mathscr{I}_\lambda^n$ the \emph{symmetric inverse semigroup of
finite transformations of the rank $\leqslant n$}. The elements of
semigroups $\mathscr{I}_\lambda^\infty$ and
$\mathscr{I}_\lambda^n$ are called \emph{finite one-to-one
transformations} (\emph{partial bijections}) of the set $X$. By
\begin{equation*}
\left(%
\begin{array}{cccc}
  x_1 & x_2 & \cdots & x_n \\
  y_1 & y_2 & \cdots & y_n \\
\end{array}%
\right)
\end{equation*}
we denote a partial one-to-one transformation which maps $x_1$
onto $y_1$, $x_2$ onto $y_2$, $\ldots$, and $x_n$ onto $y_n$, and
by $0$ the empty transformation. Obviously, in such case we have
$x_i\neq x_j$ and $y_i\neq y_j$ for $i\neq j$
($i,j=1,2,3,\ldots,n$).

Let $\lambda$ be a non-empty cardinal. On the set
 $
 B_{\lambda}=\lambda\times\lambda\cup\{ 0\}
 $,
where $0\notin\lambda\times\lambda$, we define the semigroup
operation ``$\, \cdot\, $'' as follows
\begin{equation*}
(a, b)\cdot(c, d)=
\begin{cases} (a, d), & \text{ if } b=c,\\
0, & \text{ if } b\ne c,
\end{cases}
\end{equation*}
and $(a, b)\cdot 0=0\cdot(a, b)=0\cdot 0=0$ for $a,b,c,d\in
\lambda$. The semigroup $B_{\lambda}$ is called the
\emph{semigroup of $\lambda\times\lambda$-matrix units}~(see
\cite{CP}). Obviously, for any cardinal $\lambda>0$, the semigroup
of $\lambda\times\lambda$-matrix units $B_{\lambda}$ is isomorphic
to $\mathscr{I}_\lambda^1$.

\begin{definition}[\cite{GutikPavlyk2001, Stepp1969}]\label{def1}
Let $\mathfrak{S}$ be a class of topological semigroups.
A~topological semigroup $S\in\mathfrak{S}$ is called
\emph{$H$-closed in the class $\mathfrak{S}$} if $S$ is a closed
subsemigroup of any topological semigroup $T\in{\mathfrak{S}}$
which contains $S$ as a subsemigroup. If $\mathfrak{S}$ coincides
with the class of all topological semigroups, then the semigroup
$S$ is called \emph{$H$-closed}.
\end{definition}

We remark that in~\cite{Stepp1969} $H$-closed semigroups are
called {\it maximal}.

\begin{definition}{\cite{GutikPavlyk2003, Stepp1975}}\label{def2}
Let $\mathfrak{S}$ be a class of topological semigroups.
A~topological semigroup $S\in{\mathfrak{S}}$ is called {\it
absolutely $H$-closed in the class $\mathfrak{S}$} if any
continuous homomorphic image of $S$ into $T\in\mathfrak{S}$ is
$H$-closed in $\mathfrak{S}$. If $\mathfrak{S}$ coincides with the
class of all topological semigroups, then the semigroup $S$ is
called {\it absolutely $H$-closed}.
\end{definition}

\begin{definition}{\cite{GutikPavlyk2003, Stepp1975}}\label{def3}
Let ${\mathfrak{S}}$ be a class of topological semigroups.
A~semigroup $S$ is called {\it algebraically $h$-closed in
${\mathfrak S}$} if $S$ with the discrete topology $\mathfrak{d}$
is absolutely $H$-closed in ${\mathfrak{S}}$ and $(S,
\mathfrak{d})\in{\mathfrak{S}}$. If ${\mathfrak{S}}$ coincides
with the class of all topological semigroups, then the semigroup
$S$ is called {\it algebraically h-closed}.
\end{definition}

Absolutely $H$-closed semigroups and algebraically $h$-closed
semigroups were introduced by Stepp in~\cite{Stepp1975}. There
they were called {\it absolutely maximal} and {\it algebraic
maximal}, respectively.

Gutik and Pavlyk established in \cite{GutikPavlyk2005} topological
properties of infinite topological semigroups of
$\lambda\times\lambda$-matrix units $B_{\lambda}$. They showed
that an infinite topological semigroup of
$\lambda\times\lambda$-matrix units $B_{\lambda}$ does not embed
into a compact topological semigroup, every non-zero element of
$B_{\lambda}$ is an isolated point of $B_{\lambda}$, and
$B_{\lambda}$ is algebraically h-closed in the class of
topological inverse semigroups.

Gutik, Lawson and Repov\v{s} in~\cite{GutikLawsonRepovs2007}
introduced the conception of semigroups with a tight ideal series
and there they investigated their closure in semitopological
semigroups, partially inverse semigroups with continuous
inversion. Also they derived related results about the
nonexistence of (partial) compactifications of topological
semigroups with a tight ideal series. As a corollary they show
that the symmetric inverse semigroup of finite transformations
$\mathscr{I}_\lambda^n$ of the rank $\leqslant n$ is algebraically
closed in the class of inverse (semi)topological semigroups with
continuous inversion. Since semigroups with a tight ideal series
are not preserved by homomorphisms
(\cite[Lemma~19]{GutikLawsonRepovs2007}), naturally arises the
following question: \emph{is the symmetric inverse semigroup of
finite transformations $\mathscr{I}_\lambda^n$ of the rank
$\leqslant n$ is algebraically $h$-closed in the class of
topological inverse semigroups?}

In this paper we shall show that for every infinite cardinal
$\lambda$ the finite symmetric inverse semigroup
$\mathscr{I}_\lambda^n$ of the rank $\leqslant n$ has topological
properties similar to the infinite semigroup of matrix units
$B_{\lambda}$ as a topological semigroup. We show that the
topological inverse semigroup $\mathscr{I}_\lambda^n$ is
algebraically $h$-closed in the class of topological inverse
semigroups. Also we prove that a topological semigroup $S$ with
countably compact square $S\times S$ does not contain the
semigroup $\mathscr{I}_\lambda^n$ for infinite cardinal $\lambda$
and show that the Bohr compactification of an infinite topological
symmetric inverse semigroup of finite transformations
$\mathscr{I}_\lambda^n$ of the rank $\leqslant n$ is the trivial
semigroup.

The main results of this paper were announced in
\cite{ReiterGutik2008}.

\begin{theorem}\label{theorem1}
For any positive integer $n$ the semigroup $\mathscr{I}_\lambda^n$
is algebraically $h$-closed in the class of topological inverse
semigroups.
\end{theorem}

\begin{proof}
In the case $\lambda<\omega$ the assertion of the theorem is
obvious. Suppose now that $\lambda\geqslant\omega$. We shall prove
the assertion of the theorem by induction.

Theorem~14 from~\cite{GutikPavlyk2005} implies that the semigroup
$\mathscr{I}_\lambda^1$ is algebraically $h$-closed in the class
of all topological inverse semigroups. We suppose that the
assertion of the theorem holds for $n=1,2,\ldots,k-1$ and we shall
prove that it is true for $n=k$.

Suppose to the contrary, that there exist a topological inverse
semigroup $S$ and continuous homomorphisms $h$ from the semigroup
$\mathscr{I}_\lambda^k$ with the discrete topology into $S$ such
that  $(\mathscr{I}_\lambda^k)h$ is a non-closed subsemigroup of
$S$. Since a homomorphic image of an inverse semigroup is an
inverse semigroup, Proposition~II.2 of~\cite{EberhartSelden1969}
implies that $\operatorname{cl}_S((\mathscr{I}_\lambda^k)h)$ is a
topological inverse semigroup. Therefore, without loss of
generality we can assume that $(\mathscr{I}_\lambda^k)h$ is a
dense inverse subsemigroup of $S$.

Let $x\in S\setminus(\mathscr{I}_\lambda^k)h$ and $W(x)$ be an
open neighbourhood of the point $x$. Since the semigroup
$\mathscr{I}_\lambda^{k-1}$ is algebraically $h$-closed in the
class of topological inverse semigroups, without loss of
generality we can assume that $W(x)\cap
(\mathscr{I}_\lambda^{k-1})h=\varnothing$.

Suppose that $x$ is an idempotent of $S$. Then there exists an
open neighbourhood $V(x)\subseteq W(x)$ such that $V(x)\cdot
V(x)\subseteq W(x)$. Then since the neighbourhood $V(x)$ contains
infinitely points from $(\mathscr{I}_\lambda^k)h\setminus
(\mathscr{I}_\lambda^{k-1})h$ we have that $\big(V(x)\cdot
V(x)\big)\cap(\mathscr{I}_\lambda^{k-1})h\neq\varnothing$. A
contradiction to the assumption $W(x)\cap
(\mathscr{I}_\lambda^{k-1})h=\varnothing$. There fore we have
$x\cdot x\neq x$.

Since $\mathscr{I}_\lambda^{k-1}$ is an inverse subsemigroup of
$\mathscr{I}_\lambda^{k}$ Proposition~II.2
\cite{EberhartSelden1969} implies that $x^{-1}\in
S\setminus(\mathscr{I}_\lambda^k)h$. Since $S$ is a topological
inverse semigroup and the semigroup $\mathscr{I}_\lambda^{k-1}$ is
algebraically $h$-closed in the class of topological inverse
semigroups, there exist open neighbourhoods $V(x)$ and $V(x^{-1})$
of the points $x$ and $x^{-1}$, respectively, such that
\begin{equation*}
V(x)\cdot V(x^{-1})\cdot V(x)\subseteq W(x), \quad V(x)\cap
(\mathscr{I}_\lambda^{k-1})h=\varnothing, \quad V(x^{-1})\cap
(\mathscr{I}_\lambda^{k-1})h=\varnothing,
\end{equation*}
\begin{equation*}
    \mbox{and}\qquad V(x)\subseteq W(x).
\end{equation*}
We observe that the set $V(x)\cap (\mathscr{I}_\lambda^{k})h$ is
infinite, otherwise we have that $x\not\in
\operatorname{cl}_S((\mathscr{I}_\lambda^k)h)$. Since $S$ is a
topological inverse semigroup, the set $V(x^{-1})\cap
(\mathscr{I}_\lambda^{k})h$ is infinite too. Let $V=(V(x)\cap
(\mathscr{I}_\lambda^{k})h)h^{-1}$ and $V^{\ast}=(V(x^{-1})\cap
(\mathscr{I}_\lambda^{k})h)h^{-1}$. Then the sets $V$ and
$V^{\ast}$ are infinite, and we have $V\cap
\mathscr{I}_\lambda^{k-1}=\varnothing$ and $V^{\ast}\cap
\mathscr{I}_\lambda^{k-1}=\varnothing$. Therefore $V\cdot
V^{\ast}\cdot V\cap\mathscr{I}_\lambda^{k-1}\neq\varnothing$ and
hence $((V)h\cdot (V^{\ast})h\cdot (V)h)\cap
(\mathscr{I}_\lambda^{k-1})h\neq\varnothing$. But
\begin{equation*}
((V)h\cdot (V^{\ast})h\cdot (V)h)\subseteq V(x)\cdot
V(x^{-1})\cdot V(x)\subseteq W(x),
\end{equation*}
a contradiction to the assumption $W(x)\cap
(\mathscr{I}_\lambda^{k-1})h=\varnothing$. The obtained
contradiction implies the assertion of the theorem.
\end{proof}

Theorem~\ref{theorem1} implies

\begin{corollary}\label{corollary2}
Let $n$ be any positive integer and let $\tau$ be any inverse
semigroup topology on $\mathscr{I}_\lambda^{n}$. Then
$(\mathscr{I}_\lambda^{n},\tau)$ is an absolutely $H$-closed
topological inverse semigroup in the class of topological inverse
semigroups.
\end{corollary}

The following theorem generalizes Theorem~10
from~\cite{GutikPavlyk2005}.

\begin{theorem}\label{theorem3}
A topological semigroup $S$ with countably compact square $S\times
S$ does not contain an infinite countable semigroup of matrix
units.
\end{theorem}

\begin{proof}
Suppose to the contrary: there exists a topological semigroup $S$
with countably compact square $S\times S$ such that $S$ contains
an infinite countable semigroup of $\omega\times\omega$-matrix
units $B_\omega$. We numerate elements of a set $X$ of cardinality
$\omega$ by non-negative integers, i.~e., $X=\{ \alpha_0,
\alpha_1, \alpha_2, \ldots\}$. Then we consider the sequence
$\{((\alpha_0,\alpha_n),(\alpha_n,\alpha_0))_{n=1}^\infty\}$ in
$B_\omega\times B_\omega\subset S\times S$. The countable
compactness of $S\times S$ guarantees that this sequence has an
accumulation point $(a,b)\in S\times S$. Since
$(\alpha_0,\alpha_n)\cdot(\alpha_n,\alpha_0)=(\alpha_0,\alpha_0)$,
the continuity of the semigroup operation on $S$ guarantees that
$ab=(\alpha_0,\alpha_0)$. By Lemma~4~\cite{GutikPavlyk2005}, every
non-zero element of the semigroup of $\omega\times\omega$-matrix
units $B_\omega$ endowed with the topology induced from $S$ is an
isolated point in $B_\omega$. So, there exists a neighbourhood
$O((\alpha_0,\alpha_0))\subseteq S$ of the point
$(\alpha_0,\alpha_0)\in B_\omega$ containing no other points of
the semigroup $B_\omega$. Since $ab=(\alpha_0,\alpha_0)$, the
points $a,b$ have neighborhoods $O(a), O(b)\subset S$ such that
$O(a)\cdot O(b)\subset O((\alpha_0,\alpha_0))$. Since $a$ is an
accumulation point of the sequence $(\alpha_0,\alpha_n)$, there
exists a positive integer $n$ such that $(\alpha_0,\alpha_n)\in
O(a)$. Similarly there exists a positive integer $m>n$ such that
$(\alpha_m,\alpha_0)\in O(b)$. Then
$(\alpha_0,\alpha_n)\cdot(\alpha_m,\alpha_0)=0\in O(a)\cdot
O(b)\cap B_\omega=(\alpha_0,\alpha_0)$, which is a contradiction.
\end{proof}

Since every infinite semigroup of $\lambda\times\lambda$-matrix
units $B_\lambda$ contains the semigroup $B_\omega$,
Theorem~\ref{theorem3} implies

\begin{theorem}\label{theorem3a}
A topological semigroup $S$ with countably compact square $S\times
S$ does not contain an infinite semigroup of matrix units.
\end{theorem}

Theorem~\ref{theorem3} implies

\begin{corollary}[{\cite[Theorem~10]{GutikPavlyk2005}}]\label{corollary4}
An infinite semigroup of matrix units does not embed into a
compact topological semigroup.
\end{corollary}

A semigroup homomorphism $h\colon S\rightarrow T$ is called
\emph{annihilating} if $(s)h=(t)h$ for all $s,t\in S$.

A~semigroup $S$ is called \emph{congruence-free} if it has only
two congruences: identical and universal~\cite{CP}. Obviously, a
semigroup $S$ is congruence-free if and only if every homomorphism
$h$ of $S$ into an arbitrary semigroup $T$ is an isomorphism
``into'' or is an annihilating homomorphism.

Theorem~1 from~\cite{Gluskin1955} implies that the semigroup
$B_\lambda$ is congruence-free for every cardinal
$\lambda\geqslant 2$ and hence Theorem~\ref{theorem3} implies

\begin{theorem}\label{theorem5}
Every continuous homomorphism from an infinite topological
semigroup of matrix units into a topological semigroup $S$ with
countably compact square $S\times S$ is annihilating.
\end{theorem}

Theorem~\ref{theorem5} implies

\begin{corollary}[{\cite[Theorem~12]{GutikPavlyk2005}}]\label{corollary6}
Every continuous homomorphism from an infinite topological
semigroup of matrix units into a compact topological semigroup is
annihilating.
\end{corollary}

\begin{theorem}\label{theorem7}
Let $\lambda\ge\omega$ and $n$ be a positive integer. Then every
continuous homomorphism of the topological semigroup
$\mathscr{I}_\lambda^{n}$ into a topological semigroup $S$ with
countably compact square $S\times S$ is annihilating.
\end{theorem}

\begin{proof}
We shall prove the assertion of the theorem by induction. By
Theorem~\ref{theorem5} every continuous homomorphism of the
topological semigroup $\mathscr{I}_\lambda^{1}$ into a topological
semigroup $S$ with countably compact square $S\times S$ is
annihilating.. We suppose that the assertion of the theorem holds
for $n=1,2,\ldots,k-1$ and we shall prove that it is true for
$n=k$.

Obviously it is sufficiently to show that the statement of the
theorem holds for the discrete semigroup
$\mathscr{I}_\lambda^{k}$. Let
$h\colon\mathscr{I}_\lambda^{k}\rightarrow S$ be arbitrary
homomorphism from $\mathscr{I}_\lambda^{k}$ with the discrete
topology into a topological semigroup $S$ with countably compact
square $S\times S$. Then by Theorem~\ref{theorem5} the restriction
$h_{\mathscr{I}_\lambda^{1}}\colon\mathscr{I}_\lambda^{1}
\rightarrow S$ of homomorphism $h$ onto the subsemigroup
$\mathscr{I}_\lambda^{1}$ of $\mathscr{I}_\lambda^{k}$ is an
annihilating homomorphisms. Let
$(\mathscr{I}_\lambda^{1})h_{\mathscr{I}_\lambda^{1}}=
(\mathscr{I}_\lambda^{1})h=e$, where $e\in E(S)$. We fix any
$\alpha\in\mathscr{I}_\lambda^{k}$ with
$\operatorname{ran}(\alpha)=i\geqslant 2$. Let
 $\alpha=
 \left(%
 \begin{array}{cccc}
  x_1 & x_2 & \cdots & x_i \\
  y_1 & y_2 & \cdots & y_i \\
 \end{array}%
 \right)
 $
(where $x_1, x_2,\ldots, x_i, y_1, y_2,\ldots, y_i\in X$ for some
set $X$ of cardinality $\lambda$). We fix $y_1\in X$ and define
subsemigroup $T_{y_1}$ of $\mathscr{I}_\lambda^{k}$ as follows:
\begin{equation*}
    T_{y_1}=\left\{\beta\in\mathscr{I}_\lambda^{k}\mid
      \Big(%
     \begin{array}{c}
      y_1\\
      y_1\\
     \end{array}%
      \Big)\cdot\beta=\beta\cdot
      \Big(%
     \begin{array}{c}
      y_1\\
      y_1\\
     \end{array}%
      \Big)=
      \Big(%
     \begin{array}{c}
      y_1\\
      y_1\\
     \end{array}%
      \Big)
    \right\}.
\end{equation*}
Then the semigroup $T_{y_1}$ is isomorphic to the semigroup
$\mathscr{I}_\lambda^{k-1}$, the element
 $\Big(%
     \begin{array}{c}
      y_1\\
      y_1\\
     \end{array}%
      \Big)
 $
is zero of $T_{y_1}$ and hence by induction assumption we have
$\Big(\Big(%
     \begin{array}{c}
      y_1\\
      y_1\\
     \end{array}%
      \Big)\Big)h=(\beta)h
 $
for all $\beta\in T_{y_1}$.

Since
 $\Big(%
     \begin{array}{c}
      y_1\\
      y_1\\
     \end{array}%
      \Big)\in\mathscr{I}_\lambda^{1}
 $,
we have that $(\beta)h=(0)h$ for all $\beta\in T_{y_1}$. But
$\alpha=\alpha\gamma$, where
 $\gamma=
 \left(%
 \begin{array}{cccc}
  y_1 & y_2 & \cdots & y_i \\
  y_1 & y_2 & \cdots & y_i \\
 \end{array}%
 \right)\in T_{y_1}
 $,
and hence we have
\begin{equation*}
    (\alpha)h=(\alpha\gamma)h=(\alpha)h\cdot(\gamma)h=
    (\alpha)h\cdot(0)h=(\alpha\cdot 0)h=(0)h=e.
\end{equation*}
This completes the proof of the theorem.
\end{proof}

Theorem~\ref{theorem7} implies

\begin{theorem}\label{theorem8}
Let $\lambda\ge\omega$ and $n$ be a positive integer. Then every
continuous homomorphism of the topological semigroup
$\mathscr{I}_\lambda^{n}$ into a compact topological semigroup is
annihilating.
\end{theorem}

Recall \cite{DeLeeuwGlicksberg1961} that a {\it Bohr
compactification of a topological semigroup $S$} is a~pair
$(\beta, B(S))$ such that $B(S)$ is a compact semigroup,
$\beta\colon S\to B(S)$ is a continuous homomorphism, and if
$g\colon S\to T$ is a continuous homomorphism of $S$ into a
compact semigroup $T$, then there exists a unique continuous
homomorphism $f\colon B(S)\to T$ such that the diagram
\begin{equation*}
\xymatrix{ S\ar[r]^\beta\ar[d]_g & B(S)\ar[ld]^f\\ T }
\end{equation*}
commutes.

Theorem~\ref{theorem8} and Theorem~2.44~\cite[Vol.~1]{CHK} imply

\begin{theorem}\label{theorem9}
If $\lambda\ge\omega$ and $n$ is a positive integer, then the Bohr
compactification of the topological semigroup
$\mathscr{I}_\lambda^{n}$ is a trivial semigroup.
\end{theorem}


\end{document}